\documentclass[a4paper,11pt,reqno]{amsart}

\usepackage[T1]{fontenc}
\usepackage[utf8]{inputenc}
\usepackage{tikz}
\usepackage[british]{babel}
\usepackage[shortlabels]{enumitem}
\usepackage{amsmath, amsthm, amssymb, amsfonts}
\usepackage{mathtools}
\usepackage{thmtools,thm-restate}
\usepackage{hyperref}
\usepackage{stmaryrd}
\usepackage{bbm}

\usepackage{geometry}
\declaretheorem[name=Theorem]{thm}
\declaretheorem[sibling=thm]{lemma}

\declaretheorem[sibling=thm]{claim}
\declaretheorem[sibling=thm]{conjecture}

\declaretheorem[sibling=thm,style=definition]{definition}

\theoremstyle{definition}
\newtheorem*{assert}{Assertion}

\newcommand{\bF}{\mathbf{F}}
\newcommand{\cF}{\mathcal{F}}
\newcommand{\cH}{\mathcal{H}}
\newcommand{\bH}{\mathbf{H}}
\newcommand{\cG}{\mathcal{G}}
\newcommand{\cI}{\mathcal{I}}

\newcommand{\bS}{\mathbf{S}}
\newcommand{\cS}{\mathcal{S}}

\newcommand{\cW}{\mathcal{W}}
\newcommand{\1}{\mathbbm{1}}

\renewcommand{\emptyset}{\varnothing}

\renewcommand{\le}{\leqslant}
\renewcommand{\ge}{\geqslant}
\renewcommand{\leq}{\leqslant}
\renewcommand{\geq}{\geqslant}

\newcommand{\eps}{\varepsilon}

\DeclareMathOperator{\Var}{Var}

\newcommand{\Nat}{\mathbb{N}}

\newcommand{\Ex}{\mathbb{E}}
\renewcommand{\Pr}{\mathbb{P}}

\title{Towards the Kohayakawa--Kreuter conjecture on asymmetric Ramsey properties}

\author{Frank Mousset}
\address{Frank Mousset, School of Mathematical Sciences, Tel Aviv University, Tel Aviv 6997801, Israel}
\email{moussetfrank@gmail.com}
\author{Rajko Nenadov}
\address{Rajko Nenadov, Department of Computer Science, ETH Zurich, 8092 Z\"urich, Switzerland}
\email{rnenadov@inf.ethz.ch}
\author{Wojciech Samotij}
\address{Wojciech Samotij, School of Mathematical Sciences, Tel Aviv University, Tel Aviv 6997801, Israel}
\email{samotij@post.tau.ac.il}
\thanks{Research supported in part by the Israel Science Foundation (ISF) grants 1147/14 (FM and WS) and 1028/16 (FM) and ERC Starting Grant 633509 (FM)}
\thanks{Part of the work was done while the second author was visiting Tel Aviv University.}
\thanks{A first draft of this paper was produced at the workshop of the research group of Angelika Steger in Buchboden in July 2018.}
\date{\today}

\begin{document}

\begin{abstract}
  For fixed graphs $F_1,\dotsc,F_r$, we prove an upper bound on the threshold function for the property that $G(n,p) \to (F_1,\dotsc,F_r)$. This establishes the $1$-statement of a conjecture of Kohayakawa and Kreuter.
\end{abstract}

\maketitle

\section{Introduction}

Given a graph $G$, a positive integer $r$, and graphs $F_1, \dotsc, F_r$, we  write
\[
  G \to (F_1, \dotsc, F_r)
\]
if for every colouring of the edges of $G$ using colours from the set $[r] :=
\{1,\dotsc,r\}$, there exists a copy of $F_i$ in $G$ whose all edges have been
coloured $i$, for some $i\in [r]$. We study the
asymptotic probability that
\[
  G(n,p) \to (F_1, \dotsc, F_r)
\]
for fixed graphs $F_1,\dotsc,F_r$, where $G(n,p)$ is the binomial random graph with $n$ vertices and edge probability $p$.

An important special case of this problem, known as the \emph{symmetric} case, arises when the graphs $F_1,\dotsc, F_r$ are all the same. The study of symmetric Ramsey properties in random graphs was initiated by {\L}uczak, Ruciński, and Voigt~\cite{luczak1992ramsey}, who proved that $p = n^{-1/2}$ is a threshold for the property $G(n,p) \to (K_3,K_3)$. (The earlier work of Frankl and R\"odl~\cite{FraRod86} established that $G(n,p) \to (K_3, K_3)$ under the stronger assumption that $p \ge n^{\eps - 1/2}$.) This was followed by a series of papers by R\"odl and Ruci\'nski~\cite{rodl1993lower,rodl1994random,rodl1995threshold} that culminated in the following statement. For a nonempty graph $F$, let $d_2(F) := 1/2$ if $F = K_2$ and $d_2(F):= \frac{e_{F}-1}{v_{F}-2}$ otherwise and define the \emph{$2$-density} of $F$ by
\begin{equation}\label{eq:m2}
  m_2(F) := \max{\left\{ d_2(F') : F'\subseteq F \text{ with }e_{F'}\geq 1\right\}}.
\end{equation}

\begin{thm}[\cite{rodl1995threshold}]\label{thm:rr}
  Let $r\geq 2$ and suppose that $F$ is a nonempty graph
  such that at least one component of $F$ is not a star or (in the
  case $r=2$) a path of length three.
  Then there exist positive constants $c$ and $C$ such that
  \[
    \lim_{n\to\infty} \Pr\big(G(n,p)\to (\underbrace{F,\dotsc,F}_{\text{$r$ times}})\big) =
    \begin{cases}
      0 &\text{if $p\leq cn^{-1/m_2(F)}$}\\
      1 &\text{if $p\geq Cn^{-1/m_2(F)}$}.
    \end{cases}
  \]
\end{thm}

One usually refers to the assertion of Theorem~\ref{thm:rr} for $p\leq cn^{-1/m_2(F)}$ as the \emph{$0$-statement} and to the assertion for $p\geq Cn^{-1/m_2(F)}$ as the \emph{$1$-statement}.
It is worth pointing out that the assumption on the structure of $F$ is necessary. Indeed, if every component of $F$ is a star, then it is easy to see that $G\to (F,\dotsc,F)$ as soon as $G$
has sufficiently many vertices of degree $r(\Delta_F-1)+1$. The function
$n^{-1-1/(r(\Delta_F-1)+1)}$ is a threshold for this property in $G(n,p)$; on
the other hand, $m_2(F) = 1$ for every such $F$. In the case where $r=2$ and
at least one component of $F$ is a path of length three while the others are
stars, the $0$-statement of Theorem~\ref{thm:rr} is no longer true. 
For example, if
$p=cn^{-1/m_2(P_3)}=cn^{-1}$ for some $c > 0$, then the probability that
$G(n,p)$ contains a cycle of length five with an edge pending at every vertex
is bounded from below by a positive constant (that depends on $c$); it is easy
to check that
every colouring of the edges of this graph with two colours yields a 
monochromatic path of length three.
This exceptional case,
originally missed by Rödl and Ruciński in~\cite{rodl1995threshold}, was
eventually noticed and corrected by Friedgut and
Krivelevich~\cite{friedgut2000sharp};
the corrected version of the $0$-statement
requires the assumption that $p=o(n^{-1/m_2(F)})$.
A short proof of Theorem \ref{thm:rr} was given
by Nenadov and Steger~\cite{nenadov2016short}.

In the case where $F$ is a tree (other than a star or the path of length
three)~\cite{friedgut2000sharp}, a triangle \cite{friedgut06sharp} or, more
generally, a strictly $2$-balanced\footnote{A nonempty graph $F$ is said to be
\emph{2-balanced} if $d_2(F) = m_2(F)$ and \emph{strictly $2$-balanced} if in
addition $d_2(F') < m_2(F)$ for every nonempty proper subgraph $F' \subseteq
F$.} graph that can be made bipartite by removing some edge
\cite{schacht2018sharp}, it is known that the property $G(n,p) \to (F, F)$ has
a sharp threshold: 

\begin{thm}[\cite{friedgut2000sharp,friedgut06sharp,schacht2018sharp}]\label{thm:sharp}
  Suppose that $F$ is either (i) a tree that is not a star or the path of length three
  or (ii) a strictly $2$-balanced graph with $e_F \ge 2$ edges that can be made
  bipartite by removing some edge. Then there exist $c_0$, $c_1$, and a function
  $c \colon \Nat \to [c_0, c_1]$ such that 
  \[
    \lim_{n \to \infty} \Pr\big( G(n,p) \to (F, F)\big) =
    \begin{cases}
      0 &\text{if } p \ge (1 + \eps)c(n) n^{2 - 1/m_2(F)} \\
      1 &\text{if } p \le (1 - \eps)c(n) n^{2 - 1/m_2(F)} \\
    \end{cases}
  \]
  for every positive constant $\eps$.
\end{thm}

\noindent
It is widely believed that one can choose $c(n)$ to be a constant function;
however proving this and, what is more, determining the value of the constant, remains a
formidable challenge. 

The main topic of this paper is the \emph{asymmetric} case of the Ramsey
problem in $G(n,p)$, where the graphs $F_1,\dotsc,F_r$ are allowed to be
different. 
This problem was first considered by Kohayakawa and Kreuter~\cite{KohKre97}.
For nonempty graphs $F_1$ and $F_2$ with $m_2(F_1)\geq m_2(F_2)$, we define the \emph{asymmetric $2$-density}
\begin{equation}
  \label{eq:m2-F1-F2}
  m_2(F_1, F_2) := \max\left\{ \frac{e_{F_1'}}{v_{F_1'} - 2 + 1/m_2(F_2)} :
  F_1' \subseteq F_1 \text{ with } e_{F_1'} \geq 1 \right\}.
\end{equation}
The following generalisation of Theorem~\ref{thm:rr} is a slight rephrasing of a conjecture made by Kohayakawa and Kreuter~\cite{KohKre97}. (The original conjecture was stated only for two colours and it lacked the assumption that the graphs $F_1$ and $F_2$ are not forests, which was later added by Kohayakawa, Schacht, and Sp\"ohel~\cite{kohayakawa2014upper}.)

\begin{conjecture} \label{conj:kk}
  Let $r\geq 2$ and suppose that $F_1, \dotsc, F_r$ are graphs with $m_2(F_1) \geq \dotsb \geq m_2(F_r)$ and $m_2(F_2) > 1$. Then there are positive constants $c$ and $C$ such that
  \[
    \lim_{n\to\infty} \Pr\big(G(n,p)\to (F_1,\dotsc,F_r)\big) =
    \begin{cases}
      0 &\text{if $p\leq c n^{-1/m_2(F_1,F_2)}$}\\
      1 &\text{if $p\geq C n^{-1/m_2(F_1,F_2)}$}.
    \end{cases}
  \]
\end{conjecture}

The assumption that $m_2(F_2) \geq 1$ is necessary, since otherwise we have $m_2(F_2)=1/2$ (i.e., $F_2$ is a matching) and so $m_2(F_1,F_2) = e_{F_1'}/v_{F_1'}$ for some nonempty subgraph $F'_1\subseteq F_1$. In this case, for every constant $C > 0$, the probability that $G(n,p)$ with $p = Cn^{-1/m_2(F_1,F_2)}$ contains no copies of $F_1$ at all exceeds a positive constant (that depends on $C$); see, for example~\cite{JanLucRuc00book}. The assumption that $m_2(F_2) > 1$ (which holds if and only if $F_2$ is not a forest) is most likely not always necessary, but it precludes exceptional sequences $F_1, \dotsc, F_r$ such as a sequence of stars.

There have been several attempts at resolving Conjecture~\ref{conj:kk}. Kohayakawa and Kreuter~\cite{KohKre97} proved it in the case where each $F_i$ is a cycle. Marciniszyn, Skokan, Sp\"ohel, and Steger~\cite{marciniszyn2009asymmetric} observed that the proof of the $1$-statement of Conjecture~\ref{conj:kk} given in~\cite{KohKre97} for sequences of cycles extends to all sequences $F_1, \ldots, F_r$ such that $F_1$ is strictly $2$-balanced, assuming the so-called K{\L}R (Kohayakawa--{\L}uczak--R\"odl) conjecture~\cite{kohayakawa1997klr} holds. (The K{\L}R conjecture has since been verified, see~\cite{balogh2015independent,conlon2014klr,saxton2015hypergraph}.) The main result of~\cite{marciniszyn2009asymmetric} however was a proof of the $0$-statement of Conjecture~\ref{conj:kk} in the case where each $F_i$ is a complete graph. A self-contained (i.e., not relying on the K{\L}R conjecture) proof of the $1$-statement of Conjecture~\ref{conj:kk} for $r = 2$ that assumes a similar density condition, namely that $F_1$ is strictly balanced w.r.t.\ $m_2(\cdot, F_2)$, was given by Kohayakawa, Schacht, and Sp\"ohel~\cite{kohayakawa2014upper}.
This result was generalised by allowing $F_1, \dotsc, F_r$ to be uniform hypergraphs and extended from two to an arbitrary number of colours by Gugelmann, Nenadov, Person, \v{S}kori\'c, Steger, and Thomas~\cite{gugelmann2017symmetric}. It was furthermore shown in~\cite{gugelmann2017symmetric} that the $1$-statement of Conjecture~\ref{conj:kk} holds, with no additional conditions on the graphs $F_1,\dotsc,F_r$, under the stronger assumption that $p \ge Cn^{-1/m_2(F_1,F_2)}\log n$. The proofs of both these statements employed the hypergraph container method developed by Balogh, Morris, and Samotij~\cite{balogh2015independent} (see also~\cite{BalMorSam-ICM}) and, independently, by Saxton and Thomason~\cite{saxton2015hypergraph}. Let us also mention that the short argument of Nenadov and Steger~\cite{nenadov2016short} that establishes the $1$-statement of Theorem~\ref{thm:rr} can be rewritten almost verbatim to give a proof of the $1$-statement of Conjecture~\ref{conj:kk} in the case where $m_2(F_1) = m_2(F_2)$, with no further conditions on $F_1$ (this was explicitly observed in~\cite{hancock2017independent}).

In summary, all previous results related to the $1$-statement of Conjecture~\ref{conj:kk} require either some nontrivial assumptions on $F_1$ (or both $F_1$ and $F_2$) or a stronger lower bound on $p$. Our main contribution is a proof of this statement in its full generality.

\begin{thm}\label{thm:main}
  Let $r\geq 2$ and suppose that $F_1, \dotsc, F_r$ are graphs with $m_2(F_1) \ge \dotsb \ge m_2(F_r)$ and $m_2(F_2)\geq 1$. Then there exists a positive constant $K$ such that if $p=p(n)\geq Kn^{-1/m_2(F_1,F_2)}$, then
  \[
    \lim_{n\to\infty}\Pr\big(G(n,p) \to  (F_1,\dotsc,F_r)\big) = 1.
  \]
\end{thm}

\section{Preliminaries}

\subsection{Ramsey's theorem}

The following quantitative version of Ramsey's theorem can be obtained from the usual version by a standard averaging argument (see, e.g.,~\cite[Theorem 11]{gugelmann2017symmetric}).

\begin{lemma}[Ramsey's theorem]
  \label{lem:ramsey}
  For every positive integer $r$ and all graphs $F_1,\dotsc,F_r$, there exists
  a positive $\alpha$ and some $n_0$ such that the following holds for all
  $n\geq n_0$. For every colouring of the edges of $K_n$ with colours from
  $[r]$, there exists a colour $i\in [r]$ such that $K_n$ contains at least
  $\alpha n^{v_{F_i}}$ copies of $F_i$ whose edges all have colour $i$.
\end{lemma}

\subsection{Hypergraph containers}

The following lemma is a well-known consequence of the hypergraph container theorems obtained by Balogh, Morris, and
Samotij~\cite{balogh2015independent} and, independently, by Saxton and Thomason~\cite{saxton2015hypergraph}.

\begin{lemma}
  \label{lem:containers}
  For every graph $F$ and every positive $\eps$, there exists a positive $C=C(F,\eps)$ such that the following holds for
  all $n\in\Nat$. Let $\mathbf F(n)$
  be the family of all $F$-free graphs with vertex set $[n]$.
  Then there exist functions
  \[
    g\colon \mathbf F(n) \to 2^{E(K_n)}
    \quad\text{and}\quad
    f\colon 2^{E(K_n)} \to 2^{E(K_n)}
  \]
  such that, for every $G\in \mathbf F(n)$,
  \begin{enumerate}[label=(\roman*)]
    \item
      $g(G)$ has at most $Cn^{2-1/m_2(F)}$ edges,
    \item
      $f(g(G))$ contains at most $\eps n^{v_F}$ copies of $F$, and
    \item
      $g(G) \subseteq G \subseteq f(g(G))$.
  \end{enumerate}
\end{lemma}

\section{Proof overview}

Suppose that $G \nrightarrow (F_1, \dotsc, F_r)$, that is, that there exists a
colouring $c \colon E(G) \to [r]$ such that, for each $i \in [r]$, the graph
$c^{-1}(i)$ of edges coloured $i$ is $F_i$-free.
It follows from
Lemma~\ref{lem:containers} that there are `signatures' $S_2 =
g_2(c^{-1}(2)),\dotsc, S_r = g_r(c^{-1}(r))$, each with $|S_i|= O\big(n^{2-1/m_2(F_i)}\big)
\leq O\big(n^{2-1/m_2(F_2)}\big)$ edges, such that
$S_i \subseteq c^{-1}(i) \subseteq f_i(S_i)$ for each $i \in \{2, \dotsc, r\}$,
where the graph $f_i(S_i)$ contains $o(n^{v_{F_i}})$ copies of $F_i$. Since each
edge of $G$ that lies outside $f_2(S_2) \cup \dotsb \cup f_r(S_r)$ is coloured
$1$, the intersection of $G$ and the graph \[
  K(S_2, \dotsc, S_r) := K_n \setminus \big(f_2(S_2) \cup \dotsb \cup f_r(S_r)\big)
\]
must be $F_1$-free. In particular, the event $G(n,p) \nrightarrow (F_1, \dotsc,
F_r)$ is contained in the union of the
events
\[
  S_2 \cup \dotsb \cup S_r \subseteq G(n,p) \quad \text{and} \quad G(n,p) \cap K(S_2, \dotsc, S_r)
  \text{ is $F_1$-free},
\]
where $(S_2, \dotsc, S_r)$ ranges over all sequences of `signatures'. Since the
property ``$S_2 \cup \dotsb \cup S_r \subseteq G$\,'' is increasing in $G$ and the
property ``$G \cap K(S_1,\dotsc,S_r)$ is $F_1$-free'' is decreasing in $G$, Harris's
inequality and the union bound yield
\[
  \Pr\big(G(n,p) \nrightarrow (F_1, \dotsc, F_r)\big) \leq\kern-7pt \sum_{(S_2, \dotsc,
  S_r)} \Pr\big(S_2 \cup \dotsb \cup S_r \subseteq G(n,p)\big) \cdot \Pr\big(
  G(n,p) \cap K(S_2, \dotsc, S_r) \text{ is $F_1$-free}\big).
\]
Lemma~\ref{lem:ramsey} implies that the graph $K(S_2, \dotsc, S_r)$ has at least $\delta n^{v_{F_1}}$ copies of $F_1$, for some constant $\delta > 0$ that is independent of $(S_2, \dotsc, S_r)$, and consequently, following~\cite{JanLucRuc90}, one can derive the bound
\begin{equation}
  \label{eq:Janson-proof-overview}
  \Pr\big(G(n,p) \cap K(S_2, \dotsc, S_r)\text{ is $F_1$-free}\big) \le \exp\left(-\delta' \cdot \min\left\{n^{v_I} p^{e_I} : \emptyset \neq I \subseteq F_1 \right\}\right)
\end{equation}
from Janson's inequality.

If one assumes that $p \gg n^{-1/m_2(F_1, F_2)} \log n$, then the right-hand
side of~\eqref{eq:Janson-proof-overview} can be bounded from above by
$\exp\left(-\omega\big(n^{2-1/m_2(F_2)} \log n\big)\right)$, whereas there are
only $\exp\left(O\big(n^{2-1/m_2(F_2)} \log n\big)\right)$ sequences $(S_2,
\dotsc, S_r)$; thus one may conclude that $G(n,p) \rightarrow (F_1, \dotsc,
F_r)$ with probability very close to $1$, without any further assumptions on
$(F_1,
\dotsc, F_r)$. The weaker assumption that $p \gg n^{-1/m_2(F_1, F_2)}$ implies
only the upper bound $\exp\left(-\omega\big(n^{2-1/m_2(F_2)}\big)\right)$ on
the right-hand side of~\eqref{eq:Janson-proof-overview} and the challenge is
to obtain the estimate
\[
  \sum_{(S_2, \dotsc, S_r)} \Pr\big(S_2 \cup \dotsb \cup S_r \subseteq G(n,p)\big) = \exp\left(O\big(n^{2-1/m_2(F_2)}\big)\right).
\]
Unfortunately, this estimate is valid only if $p =
O\big(n^{-1/m_2(F_2)}\big)$, which we can assume only if $m_2(F_1) =
m_2(F_2)$. This is the essence of why proving the $1$-statement of
Conjecture~\ref{conj:kk} is significantly more difficult than proving the
$1$-statement of Theorem~\ref{thm:rr}.

A first step towards making the above union bound argument more efficient, already taken in~\cite{gugelmann2017symmetric}, is to restrict the family of `non-Ramsey' colourings that are being considered. To this end, observe that every colouring $c \colon E(G) \to [r]$ such that
\begin{equation}
  \label{eq:non-Ramsey-colouring}
  \text{$c^{-1}(i)$ is $F_i$-free for every $i \in [r]$}
\end{equation}
may be altered as follows: Every edge of $G$ that is not contained in a copy
of $F_1$ in $G$ is recoloured $1$. This way we obtain a new colouring $c$ that
still satisfies~\eqref{eq:non-Ramsey-colouring} but now each edge of
$c^{-1}(2) \cup \dotsb \cup c^{-1}(r)$ lies in a copy of $F_1$ in $G$. We may
thus replace the event ``$S_2 \cup \dotsb \cup S_r \subseteq G(n,p)$'' in the
above argument with the event ``each edge of $S_2 \cup \dotsb \cup S_r$ lies
in a copy of $F_1$ in $G(n,p)$'', which we will abbreviate by $S_2 \cup \dotsb
\cup S_r \subseteq_{F_1} G(n,p)$, and conclude that the event $G(n,p) \nrightarrow
(F_1, \dotsc, F_r)$ is contained in the union of the events
\[
  S_2 \cup \dotsb \cup S_r \subseteq_{F_1} G(n,p) \quad \text{and} \quad G(n,p) \cap K(S_2, \dotsc, S_r) \text{ is $F_1$-free},
\]
where again $(S_2, \dotsc, S_r)$ ranges over all sequences of `signatures'. Since the property $S_2 \cup \dotsb \cup S_r \subseteq_{F_1} G$ is still increasing in $G$, in order to complete the argument, it would suffice to prove that
\begin{equation}
  \label{eq:ex-signatures-overview}
  \sum_{(S_2, \dotsc, S_r)} \Pr\big(S_2 \cup \dotsb \cup S_r \subseteq_{F_1}
  G(n,p)\big) = \exp\left(O\big(n^{2-1/m_2(F_2)}\big)\right).
\end{equation}
Unfortunately, nothing in the spirit of~\eqref{eq:ex-signatures-overview} can
be true in general. Indeed, there are pairs of graphs $F_1, F_2$ such that
$m_2(F_1) > m_2(F_1, F_2) > m_2(F_2)$, but typically \emph{every} edge of
$G(n,p)$ lies in a copy of $F_1$
(for example, 
this is the case when $F_1$ contains an isolated edge), which
contradicts~\eqref{eq:ex-signatures-overview}. However, in the case where $F_1$
is strictly balanced w.r.t.\ $m_2(\cdot, F_2)$, one can prove a
version\footnote{When $m_2(F_1) > m_2(F_2)$, even if one assumes that $F_1$ is
strictly balanced w.r.t.\ $m_2(.,F_2)$,
  there is an event of very small probability in $G(n,p)$ that nevertheless
  blows up the left-hand side of~\eqref{eq:ex-signatures-overview} above
  $\exp\left(O\big(n^{2-1/m_2(F_2)}\big)\right)$. However, after conditioning
  on the complement of this event, the
  inequality~\eqref{eq:ex-signatures-overview} becomes true.}
of~\eqref{eq:ex-signatures-overview} that complements an argument similar to
the one we have outlined above; this was achieved
in~\cite{gugelmann2017symmetric}.

In order to dispose of the balancedness assumption, we shall restrict our attention only to a
subcollection of all copies of $F_1$ in $K_n$. More precisely, we will
require $c^{-1}(1)$ to avoid only a certain family $\cF$ of copies of $F_1$,
which we term \emph{typed} copies of $F_1$. Roughly speaking, every edge of~$K_n$
will be (randomly) mapped to an edge of $F_1$, which we call a type, and a copy
$\varphi(F_1)$ of $F_1$ in $K_n$ will belong to $\cF$ if and only if, for each $e \in E(F_1)$,
the type of $\varphi(e)$ is $e$. Moreover, we will replace
$G(n,p)$ with a sparser random subgraph of it, whose edge probabilities depend on the types.
(We will define both these
notions formally in Section~\ref{sec:typed-graphs}.) We are going to do both of these
things in such a way that the typed copies of $F_1$ in this sparser random graph `behave'
like (untyped) copies of a graph that is balanced w.r.t.\ $m_2(\cdot, F_2)$. In particular: (i)~the left-hand side
of~\eqref{eq:Janson-proof-overview} can be still bounded from above by
$\exp\left(-\omega\big(n^{2-1/m_2(F_2)}\big)\right)$ and (ii)~a version
of~\eqref{eq:ex-signatures-overview} holds true. Since the left-hand side
of~\eqref{eq:Janson-proof-overview} is decreasing in both $G(n,p)$ and the
family $\cF$, whereas each term in the left-hand side
of~\eqref{eq:ex-signatures-overview} is increasing in both $G(n,p)$ and $\cF$, we
will have to strike a delicate balance in order to achieve (i) and (ii)
simultaneously.

\section{$H$-typed graphs}
\label{sec:typed-graphs}

For the entirety of this section, suppose that $H$ and $F$ are fixed nonempty
graphs.

An important role in our proof is played by weight functions $w\colon E(H) \to [1,\infty)$. Given such a~function $w$ and a subgraph $I\subseteq H$, we shall use the shorthand $w_I := \sum_{e\in E(I)} w(e)$.

\begin{definition}[$(w,F)$-balanced graphs]\label{def:balance}
  Given a function $w\colon E(H)\to [1,\infty)$, we say that $H$ is \emph{$(w,F)$-balanced} if, for every edge $e\in E(H)$, we have
  \[
    \min\big\{ v_I-w_I/m_2(H,F) : I \subseteq H \text{ with } e \in E(I)\big\} = 2 - 1/m_2(F).
  \]
\end{definition}

\begin{lemma}\label{lem:balance}
  If $m_2(H)\geq m_2(F)$, then there exists a function $w\colon   E(H)\to[1,\infty)$ such that $H$ is $(w,F)$-balanced.
\end{lemma}
\begin{proof}
  Given a function $w \colon E(H) \to [1,\infty)$ and an edge $e\in E(H)$, let us write
  \[
    r_e(w) := \min\big\{ v_I-w_I/m_2(H,F) : I \subseteq H \text{ with } e \in E(I)\big\} - 2 + 1/m_2(F).
  \]
  Our goal is then to show that there is a function $w$ such that $r_e(w) = 0$
  for all $e \in E(H)$. To this end, consider the set $\cW$ of all
  functions $w\colon E(H)\to[1,\infty)$ that satisfy $r_e(w) \ge 0$ for all $e \in E(H)$. Note that
  the constant function $w \equiv 1$ belongs to $\cW$, by the definition of
  $m_2(H, F)$,
  and that $0 \le r_e(w) \le 1/m_2(F) - w(e)/m_2(H,F)$ for every $e \in E(H)$
  and all $w\in\cW$. In particular,
  $\cW$ is nonempty and compact and thus the
  (continuous) map $\cW \ni w \mapsto w_H \in [0, \infty)$ achieves its maximum at
  some $\hat{w} \in \cW$.
  
  We claim that $r_e(\hat{w}) = 0$ for each $e \in
  E(H)$. If this were not true and there was an $e \in E(H)$ satisfying
  $r_e(\hat{w}) > 0$, then, for some sufficiently small $\eps > 0$, the
  function $\tilde{w}$ defined by $\tilde{w}(f) = \hat{w}(f) + \eps \cdot
  \1[f=e]$ would belong to $\cW$, contradicting the maximality of $\hat{w}$.
\end{proof}

\subsection{$H$-typed graphs}

By an \emph{$H$-typed} graph we mean a graph $\cG$ equipped with a \emph{type function} $\tau_\cG\colon E(\cG)\to E(H)$. If $H$ is clear from the context, we shall just say that $\cG$ is a \emph{typed graph}. We shall mostly use calligraphic letters to denote typed graphs; however, every subgraph $I\subseteq H$ will be treated as a typed graph in the natural way, by taking its type function to be the inclusion map from $E(I)$ into $E(H)$. We write $\cG'\cong \cG$ if there is a graph isomorphism between $\cG'$ and $\cG$ that preserves the type of every edge (in this case, we say that $\cG$ and $\cG'$ are \emph{typomorphic}). We write $\cG'\preceq\cG$ if $\cG'\subseteq \cG$ and $\tau_{\cG'}(e) = \tau_\cG(e)$ for all $e\in E(\cG')$. A \emph{typed copy} of a subgraph $I \subseteq H$ in $\cG$ is a typed graph $\cI \preceq \cG$ such that $\cI \cong I$, where $I$ is treated as a typed graph. Finally, for a set $\bH$ of (untyped) copies of $H$ in $K_n$ and a typed graph $\cG$ on the vertex set $[n]$, we write $\bH(\cG)$ for the set of all $\tilde H\in\bH$ such that $E(\tilde H) \subseteq E(\cG)$ and the typed graph $\tilde \cH$ obtained by equipping $\tilde H$ with the type function $\tau_\cG|_{E(\tilde H)}$ is typomorphic to $H$.

\subsection{Random $H$-typed graphs}

We define a random $H$-typed graph $\cG(n,p,w)$ as follows.

\begin{definition}[$\cG(n,p,w)$]
  Given $n\in \Nat$, $p\in (0,1)$, and $w\colon E(H)\to [1,\infty)$, we define
  $\cG(n,p,w)$ to be the random $H$-typed graph $\cG$ on the vertex set $[n]$
  constructed by the following two-step procedure:
  \begin{enumerate}[label=(\roman*)]
  \item
    Choose a function $\tau \colon E(K_n) \to E(H)$ uniformly at random.
  \item
    Include every $e \in E(K_n)$ independently with probability $p^{w(\tau(e))}$ and set $\tau_\cG = \tau |_{E(\cG)}$.
  \end{enumerate}
\end{definition}

We shall be using the following estimate of the upper tail of the number of typed copies of $H$ and its subgraphs in $\cG(n,p,w)$; the proof is a straightforward modification of the classical argument of Ruci\'nski and Vince~\cite{RuVi85}.

\begin{lemma}\label{lem:uppertail}
  Fix a nonempty subgraph $I\subseteq H$ and let $X_I$ denote the number of typed copies of $I$ in $\cG(n,p,w)$. We have
  \begin{equation} \Pr\big(X_I\geq 2\Ex[X_I]\big)
    \leq c\cdot \left(\min_{\emptyset \neq I'\subseteq I}
    n^{v_{I'}}p^{w_{I'}}\right)^{-1},
  \end{equation}
  for some positive constant $c$ depending only on $H$.
\end{lemma}
\begin{proof}
  It is easy to see that
  \[ \Ex[X_I] = \Theta\left(n^{v_I} p^{w_I}\right) \]
  and
  \[ \Var[X_I] = O\left(\max_{\emptyset \neq I'\subseteq I}
  {n^{2v_I-v_{I'}} p^{2w_I-w_{I'}}}\right)
  = O\left(\frac{\Ex[X_I]^2}{\min_{\emptyset \neq I'\subseteq I}
  n^{v_{I'}}p^{w_{I'}}}\right). \]
  The assertion now follows from Chebyshev's inequality
  $\Pr\big(X_I\geq 2\Ex[X_I]\big) \leq \Var[X_I]/\Ex[X_I]^2$.
\end{proof}

\begin{lemma}\label{lem:janson}
  Fix a positive $\alpha$
  and a family $\bH$ of at least $\alpha
  n^{v_H}$ copies of $H$ in $K_n$.
  Then
  \begin{equation} 
    \Pr\big(\bH(\cG(n,p,w)) = \emptyset\big)
    \leq
    \exp\left(-
    c \cdot \min_{\emptyset \neq I \subseteq H} n^{v_{I}} p^{w_I}\right),
  \end{equation}
  for some positive constant $c$ depending only on $H$ and $\alpha$.
\end{lemma}

\begin{proof}
  For a given copy $C \in \bH$, let us write $\1_C$ for the indicator variable
  of the event that $C\in \bH(\cG(n,p,w))$. Thus $|\bH(\cG(n,p,w))| = \sum_{C \in \bH} \1_C$.
  Observe that $X_C$ and $X_{C'}$ are independent if $C$ and $C'$ are edge-disjoint.

  For every $C\in \bH$, we have $\Ex[\1_C] \ge p^{w_H}/e_H^{e_H}$, and so
  \[
    \mu := \Ex[|\bH(G(n,p,w))|] \ge |\bH|p^{w_H}/e_H^{e_H} \geq \alpha n^{v_H} p^{w_H}/e_H^{e_H}.
  \]
  Define
  \[
    \Delta := \sum_{\substack{\{C,C'\}\subseteq \bH\\E(C\cap C')\neq \emptyset}}
    \Ex[\1_C\1_{C'}]\qquad \text{and}\qquad \delta := \max_{C\in
    \bH}\sum_{\substack{C' \in \bH\\E(C'\cap C)\neq \emptyset}} \Ex[\1_{C'}],
  \]
  It is easy to check that
  \[ \Delta = O\left(\frac{\mu^2}{\min_{\emptyset \neq I \subsetneq H}
  n^{v_I}p^{w_I}}\right), \] where the constants implicit in the
  $O$-notation may depend on $\alpha$ and on $H$. It is also easy to see that
  \[
    \delta \leq e_H n^{v_H-2} p^{w_H}.
  \]
  The claim then follows from the following version of Suen's inequality due to
  Janson~\cite[Theorem~3]{janson1998new}:
  \[ \Pr\big(|\bH(G(n,p,w))| = 0\big) \leq
  \exp\left(-\min\left(\frac{\mu^2}{8\Delta},\frac{\mu}{6\delta},
  \frac{\mu}{2}\right)\right). \]
  Note in particular that $\mu/\delta \geq \alpha n^2/e_H^{e_H}=
  \Omega(n^{v_I}p^{w_I})$ for
  every subgraph $I\subseteq H$ consisting of two vertices and one edge.
\end{proof}


\section{Proof of Theorem~\ref{thm:main}}

Fix nonempty graphs $F_1,\dotsc,F_r$ with $m_2(F_1)\geq \dotsb \geq m_2(F_r)$ and $m_2(F_2)\geq 1$. For the sake of brevity, we  shall write $H = F_1$ and $F=F_2$.
For each $i \in \{2, \dotsc, r\}$, let $C_i$, $f_i$, and $g_i$ be as given by Lemma~\ref{lem:containers} applied to the graph $F_i$ and to some sufficiently small positive constant $\eps=\eps(r,F_1,\dotsc,F_r)$; let $C := \max{\{C_2,\dotsc,C_r\}}$.

Given a typed graph $\cG$ on $[n]$, we write $\cG\nrightarrow (H,F_2,\dotsc,F_r)$ if there exists a colouring $c \colon E(\cG) \to [r]$ such that there is neither a typed copy of $H$ in colour $1$, nor an (untyped) copy of $F_i$ in colour $i$, for any $i \in \{2, \dotsc, r\}$. Note that if $\cG \nrightarrow (H, F_2, \dotsc, F_r)$, then there also exists such a colouring $c$ where additionally every edge of $\cG$ that is not contained in a typed copy of $H$ has colour $1$. Indeed, we can always recolour such edges in colour $1$ without creating any copies of $H$ in $c^{-1}(1)$.

Now, assume that $\cG \nrightarrow (H, F_2, \dotsc, F_r)$ and let $c \colon E(\cG) \to [r]$ be a colouring that satisfies all of the above conditions. For each $i \in \{2, \dotsc, r\}$, let $S_i := g_i\big(c^{-1}(i)\big)$ and recall that $S_i$ has at most $C n^{2-1/m_2(F)}$ edges. By the definition of $f_i$ and $g_i$, we have $S_i\subseteq c^{-1}(i) \subseteq f_i(S_i)$ for every $i \in \{2, \dotsc, r\}$. Set $\bS := (S_2,\dotsc,S_r)$ and let $\bH_{\bS}$ be the collection of all (untyped) copies of $H$ in the graph $K_n \setminus \big(f_2(S_2) \cup \dotsb \cup f_r(S_r) \big)$. Observe crucially that $\bH_{\bS}(\cG) = \emptyset$, because every typed copy of $H$ in $\bH_{\bS}(\cG)$ would have all of its edges coloured $1$, contradicting the choice of the colouring.

By Lemma~\ref{lem:balance}, we can choose a function $w\colon E(H)\to[1,\infty)$ such that $H$ is $(w,F)$-balanced. In particular, for every edge $e\in E(H)$, we can fix a subgraph $H_e\subseteq H$ containing $e$ such that 
\[
  v_{H_e}-w_{H_e}/m_2(H,F) = 2-1/m_2(F).
\]
Since every edge in $\cG$ of a colour different from $1$ is in a typed copy of $H$ in $\cG$, we can conclude that for every edge $e \in S_2\cup \dotsb \cup S_r$, there is a typed copy of $H_f$ in $\cG$ which contains the edge $e$, where $f = \tau_\cG(e)$. Since $| S_2\cup \dotsb \cup S_r |\leq (r-1)Cn^{2-1/m_2(F)}$, the union of these typed copies is a typed subgraph of $\cG$ with at most $T(n) := e_H(r-1)Cn^{2-1/m_2(F)}$ edges.

Let us now summarise the above discussion. For each $i \in \{2, \dotsc, r\}$,
let $\cS_i$ comprise the family of all sets of the form $g_i(G)$ where $G$ is
an $F_i$-free graph on $[n]$ and let $\cS = \cS_2 \times \dotsb \times \cS_r$.
Moreover, let $\mathfrak W_{n}$ be the collection of all typed graphs $\cW$
with $V(\cW) \subseteq [n]$ and $e(\cW) \le T(n)$ such that every edge $e\in
E(\cW)$ is contained in a typed copy of $H_f$ in $\cW$ for some $f \in E(H)$
(where it is not necessarily the case that $f = \tau_{\cW}(e)$). What we have
shown above can be phrased as follows.

\begin{assert}
  If $\cG$ is a typed graph on $[n]$ such that $\cG \nrightarrow
  (H,F_2,\dotsc,F_r)$, then there exists a sequence $\bS = (S_2,\dotsc,S_r) \in
  \cS$, with $S_2, \dotsc, S_r$ pairwise disjoint, and some $\cW\in \mathfrak
  W_{n}$ such that:
  \begin{enumerate}[label=(\arabic*)]
  \item
    \label{item:wtiness-1}
    $S_2\cup \dotsb \cup S_r \subseteq \cW$ and $\cW\preceq \cG$ and
  \item
    \label{item:witness-2}
    $\bH_{\bS}(\cG) = \emptyset$.
  \end{enumerate}
  We call such a pair $(\bS, \cW)$ a \emph{witness} for the fact that $\cG\nrightarrow (H,F_2,\dotsc,F_r)$.
\end{assert}

Now, suppose that $p \ge Kn^{-1/m_2(H,F)}$ for a sufficiently large constant $K$. Our goal is to prove that
\[
  \Pr\big(G(n,p) \nrightarrow (H, F_2, \dotsc, F_r)\big) = o(1).
\]
As the property $G \nrightarrow (H, F_2, \dotsc, F_r)$ is monotone decreasing in $G$, we
may assume without loss of generality that $p = Kn^{-1/m_2(H,F)}$. Since there
is a natural coupling of $G(n,p)$ and $\cG(n,p,w)$ such that
$E(\cG(n,p,w))\subseteq E(G(n,p))$, we may conclude that
\[ \Pr\big(G(n,p) \nrightarrow (H,F_2,\dotsc,F_r)\big) \leq \Pr\big(\cG(n,p,w)
\nrightarrow (H,F_2,\dotsc,F_r)\big). \]
In particular, it suffices to show that the
probability that $\cG(n,p,w)$ admits a witness $(\bS, \cW)$ for the fact that $\cG(n,p,w)
\nrightarrow (H, F_2, \dotsc, F_r)$ is small. This is what we are going to do in the
remainder of the proof.

Let $\cG\sim \cG(n,p,w)$. For a subgraph $I\subseteq H$, we write $X_I$ for the
number of typed copies of $I$ contained in $\cG$. We shall now split the proof
into two cases, depending on whether or not $X_I$ exceeds $2\Ex[X_I]$ for some
nonempty $I\subseteq H$.

\medskip
\noindent
\textbf{Case 1.} There is a nonempty subgraph $I\subseteq H$ such that $\cG$ has more than $2\Ex[X_I]$ typed copies of $I$.

\smallskip
The probability that $\cG$ is in this case tends to zero by
Lemma~\ref{lem:uppertail}, because $n^{v_I}p^{w_I} \geq Kn^{2-1/m_2(F)}\to
\infty$ for every nonempty $I\subseteq H$, by the definition of
$(w,F)$-balancedness.

\medskip
\noindent
\textbf{Case 2.} For every nonempty subgraph $I\subseteq H$, $\cG$ contains at most $2\Ex[X_I]$ typed copies of $I$.

\smallskip
Let us write $\mathfrak U_{n}\subseteq \mathfrak W_{n}$ for the subset comprising all $\cW\in \mathfrak W_{n}$ that contain at most $2\Ex[X_I]$ typed copies of every nonempty $I\subseteq H$. Note that if $(\bS, \cW)$ is a witness for $\cG \nrightarrow (H, F_2, \dotsc, F_r)$, then necessarily $\cW \in \mathfrak{U}_n$, since otherwise $\cG$ would fall into the first case. Let $Z$ denote the number of witnesses $(\mathbf S,\cW)$ with $\cW\in \mathfrak U_{n}$. We shall show that $\Ex[Z] = o(1)$, which, by Markov's inequality, will imply that the probability that $\cG \nrightarrow (H, F_2, \dotsc, F_r)$ tends to zero.

To this end, we have
\[ 
  \Ex[Z]\leq \sum _{\cW\in \mathfrak U_n} \sum_{\bS \in \cS(\cW)} \Pr\big((\mathbf S,\cW)\text{ is a witness for $\cG \nrightarrow (H,F_2,\dotsc,F_r)$}\big),
\]
where we write $\cS(\cW)$ for the set of all sequences $\bS = (S_2, \dotsc, S_r) \in \cS$ such that
$S_2,\dotsc,S_r$ are pairwise edge-disjoint and
$S_2\cup \dotsb \cup S_r\subseteq \cW$
(in particular, note that $|\cS(\cW)| \leq r^{e(\cW)}$). If a pair $(\bS,\cW)$ is a witness for $\cG \nrightarrow (H,F_2,\dotsc,F_r)$, then $\cW\preceq \cG$ and $\bH_{\bS}(\cG) = \emptyset$. Thus
\[
  \Ex[Z] \leq \sum _{\cW\in \mathfrak U_n}\sum_{\bS \in \cS(\cW)} \Pr \big( \cW\preceq \cG \text{ and } \bH_{\bS}(\cG) = \emptyset \big).
\]

Given a $\cW \in \mathfrak U_n$ and an $\bS \in \cS(\cW)$, let $\bH_{\bS}^{\cW}$ denote the collection of all (untyped) copies of $H$ in $\bH_{\bS}$ that are edge-disjoint from $\cW$. Since $\bH_{\bS}^{\cW}(\cG) \subseteq \bH_{\bS}(\cG)$ and, crucially, the events $\bH_{\bS}^{\cW}(\cG) = \emptyset$ and $\cW \preceq \cG$ are independent, we obtain
\[
  \Ex[Z] \leq \sum _{\cW\in \mathfrak U_n}\sum_{\bS \in \cS(\cW)} \Pr \big(
  \cW\preceq \cG \text{ and } \bH_{\bS}^{\cW}(\cG) = \emptyset \big) = \sum
  _{\cW\in \mathfrak U_n} \Pr\big(\cW\preceq \cG\big) \sum_{\bS \in \cS(\cW)}
  \Pr\big( \bH_{\bS}^{\cW}(\cG) = \emptyset \big). \]

Since each $f_i(S_i)$ contains at most $\eps n^{v_{F_i}}$ copies of $F_i$, it follows from Ramsey's theorem (Lemma~\ref{lem:ramsey}) that if $\eps = \eps(r,H,F_2,\dotsc,F_r)$ is sufficiently small, then $|\bH_{\bS}| \geq 2 \eps n^{v_H}$. Consequently,
\[
  |\bH_{\bS}^{\cW}| \ge |\bH_{\bS}| - e(\cW) \cdot e_H \cdot n^{v_H-2} \ge 2 \eps n^{v_H} - O\big(n^{v_H - 1/m_2(F)}\big) \ge \eps n^{v_H}.
\]
It thus follows from Lemma~\ref{lem:janson} and the fact that $H$ is $(w,F)$-balanced that
\[ 
  \Pr\big(\bH_{\bS}^{\cW}(\cG) = \emptyset\big) \leq \exp\left(-
  c\cdot\min_{\emptyset\neq I\subseteq H}n^{v_{I}}p^{w_{I}}\right) \leq e^{- cK n^{2-1/m_2(F)}},
\]
for some positive constant $c$ that depends only on $H$ and $\eps$. Therefore, 
\begin{equation}
  \label{eq:janson}
  \begin{split}
    \Ex[Z] & \leq \sum _{\cW\in \mathfrak U_n} \Pr \big(\cW\preceq \cG\big)
    \sum_{\bS \in \cS(\cW)} e^{- c K n^{2-1/m_2(F)} }\\
    & \leq \sum _{\cW\in \mathfrak U_n} \Pr \big(\cW\preceq \cG\big)
    \cdot r^{e(\cW)} \cdot e^{- c K n^{2-1/m_2(F)} } \\
    & \leq
    e^{- c K n^{2-1/m_2(F)}/2 }\cdot \sum _{\cW\in \mathfrak U_n} \Pr \big(\cW\preceq \cG\big),
  \end{split}
\end{equation}
where we have used that every $\cW\in \mathfrak U_n$ has at most $T(n)$ edges and thus, for sufficiently large $K$,
\[
  e(\cW)\log r \le T(n)\log r \leq e_H(r-1)Cn^{2-1/m_2(F)}\log r \leq cKn^{2-1/m_2(F)}/2.
\]

It remains to estimate the sum in the right-hand side of~\eqref{eq:janson}. To this end, for a given $k\in\Nat$, let $\mathfrak U_{n,k}$ be the set of all $\cW\in \mathfrak U_n$ that can be written as a union $\cH_1\cup \dotsb \cup \cH_k$ of $k$ typed graphs with $V(\cH_i)\subseteq [n]$, each of which is typomorphic to $H_f$ for some $f\in E(H)$, but not as a union of fewer than $k$ such graphs. Letting $U_k$ count the number of $\cW \in \mathfrak U_{n,k}$ such that $\cW \preceq \cG$, we now have
\begin{equation}\label{eq:bla}
  \sum_{\cW\in \mathfrak U_n} \Pr\big(\cW\preceq \cG\big)
  = \sum_{k \le T(n)} \sum_{\cW\in \mathfrak U_{n,k}} \Pr\big(\cW\preceq \cG\big)
  =  \sum_{k \le T(n)} \Ex[ U_k ].
\end{equation}

\begin{claim}
  There is a constant $c_H$ that depends only on $H$ such that, for every $k$,
  \[
    \Ex[U_k] \le \left( \frac{c_H \cdot K^{w_H} n^{2 - 1/m_2(F)} }{k}\right)^k.
  \]
\end{claim}

\begin{proof}[Proof of Claim]
  Since the only member of $\mathfrak U_{n,0}$ is the empty graph, we have $U_0 = 1$. We now show that, for every $k \ge 1$,
  \begin{equation}
    \label{eq:induction}
    \Ex[U_k] \le \frac{1}{k} \cdot \Ex[U_{k-1}] \cdot e_{H} 2^{1 + v_H + e_H} \cdot K^{w_H} n^{2 - 1/m_2(F)}.
  \end{equation}
   It is easy to see that~\eqref{eq:induction} implies that, for every $k$,
  \[
    \Ex[U_k] \le \frac1{k!} \left( e_H2^{1+v_H+e_H} \cdot K^{w_H} n^{2 - 1/m_2(F)} \right)^k,
  \]
  which in turn implies the assertion of the claim. Thus we only need to prove~\eqref{eq:induction}.
  
  To this end, consider some $\cW \in \mathfrak U_{n,k}$ and let $\cH_1 \cup \dotsb \cup \cH_k$ be some representation of $\cW$ as the union of typed graphs, each of which is typomorphic to some $H_f$.\footnote{Note that this representation is not necessarily unique.} Observe that, for every $i \in [k]$, the typed graph $\cW_i = \bigcup_{j \in [k]\setminus \{i\}} \cH_j$ belongs to $\mathfrak U_{n,k-1}$. Moreover, all these typed graphs are distinct, because every $\cH_i$ contains an edge that is not covered by the union of all the other $\cH_j$ for $j\neq i$. In other words, for every $\cW \in \mathfrak U_{n,k}$, there are at least\footnote{\,``At least'' since, again, $\cW$ can have more than one representation as $\cH_1 \cup \dotsb \cup \cH_k$.} $k$ distinct typed graphs $\cW_i \in \mathfrak U_{n, k - 1}$ such that $\cW = \cW_i \cup \cH_i$ for some typed copy $\cH_i$ of some $H_f$ with $f \in E(H)$.
  Denoting, for each $I \subseteq H$, the set of all typed graphs
  with vertices from $[n]$ that are typomorphic to $I$
  by $\mathfrak C_n(I)$,
  have
  \begin{equation}
    \label{eq:relation}
    k \cdot \Ex[U_k] = k \sum_{\cW \in \mathfrak U_{n,k}} \Pr\big(\cW \preceq \cG\big)
    \le \sum_{\cW' \in \mathfrak U_{n, k-1}} \sum_{f \in E(H)}  \sum_{\cH' \in \mathfrak C(H_f)} \Pr\big(\cW' \cup \cH' \preceq \cG\big),
  \end{equation}
  where we further require that the type function of each $\cH'$ in the last sum agrees with that of $\cW'$ on the intersection $\cW'\cap\cH'$ (otherwise $\cW' \cup \cH'$ would not be a well defined typed graph). Fix arbitrary $\cW' \in \mathfrak{U}_{n,k-1}$ and $f \in E(H)$ and observe that, for every $\cH' \in \mathfrak{C}(H_f)$, we have
  \[
    \Pr\big(\cW' \cup \cH' \preceq \cG\big) = \Pr\big(\cW' \preceq \cG\big) \cdot \Pr\big(\cH' \setminus (\cH' \cap \cW') \preceq \cG\big).
  \]
  By first specifying the intersection $\cI = \cW'\cap \cH'$, which is necessarily a typed copy in $\cW'$ of some $I \subseteq H_f$, we have
  \[
    \sum_{\cH' \in \mathfrak C_n(H_f)} \Pr\big(\cW' \cup \cH' \preceq \cG\big) = 
    \sum_{I\subseteq H_f}
    \sum_{\substack{\cI \in \mathfrak C_n(I)\\
    \cI \preceq \cW'}} \sum_{\substack{\cH' \in \mathfrak C_n(H_f) \\
        \cW' \cap
        \cH' = \cI}} \Pr\big(\cW' \preceq \cG\big) \cdot \Pr\big(\cH'
    \setminus \cI \preceq \cG\big),
  \]
  Because for every $\cH' \in \mathfrak C_n(H_f)$, every $I\subseteq H_f$, and
  every $\cI \in \mathfrak C_n(I)$ with $\cI \preceq \cH'$, we have
  $\Pr\big(\cH' \setminus \cI \preceq \cG\big) = p^{w_{H_f} - w_I}$, and since by the
  definition of $\mathfrak U_n$, there are at most $2\Ex[X_I] \le 2 n^{v_I}
  p^{w_I}$ typed copies of $I$ in $\cW'$, we get
  \[
    \begin{split}
      \sum_{\cH' \in \mathfrak C_n(H_f)} \Pr\big(\cW' \cup \cH' \preceq \cG\big) &\le \Pr\big(\cW' \preceq \cG\big) \cdot \sum_{I \subseteq H_f} 2 n^{v_I} p^{w_I} \cdot n^{v_{H_f} - v_I} \cdot p^{w_{H_f} - w_I} \\
      &\le \Pr\big(\cW' \preceq \cG\big) \cdot 2^{1 + v_{H_f} + e_{H_f}} \cdot n^{v_{H_f}} p^{w_{H_f}}.
    \end{split}
  \]
  Recalling that $v_{H_f} - w_{H_f} / m_2(H,F) = 2 - 1/m_2(F)$ by the definition of $H_f$, our assumption that $p = K n^{-1/m_2(H,F)}$ gives $n^{v_{H_f}} p^{w_{H_f}} = K^{w_{H_f}} n^{2-1/m_2(F)}$. We may thus conclude that
  \[
    \sum_{\cH' \in \mathfrak C_n(H_f)} \Pr\big(\cW' \cup \cH' \preceq \cG\big) \le \Pr\big(\cW' \preceq \cG\big) \cdot 2^{1 + v_{H_f} + e_{H_f}} \cdot K^{w_{H_f}} n^{2 - 1/m_2(F)}.
  \]
  Substituting this bound into~\eqref{eq:relation}, we obtain
  \[
    \begin{split}
      k \cdot \Ex[U_k] & \le \sum_{\cW' \in \mathfrak U_{n, k-1}} \sum_{f \in E(H)} \Pr\big(\cW' \preceq \cG\big)  \cdot 2^{1 + v_{H_f} + e_{H_f}} \cdot K^{w_{H_f}} n^{2 - 1/m_2(F)} \\
      & = \Ex[U_{k-1}] \cdot \sum_{f \in E(H)} 2^{1 + v_{H_f} + e_{H_f}} \cdot K^{w_{H_f}} n^{2 - 1/m_2(F)} \\
      & \le \Ex[U_{k-1}] \cdot e_H 2^{1 + v_H + e_H} \cdot K^{w_H} n^{2 - 1/m_2(F)}.
    \end{split}
  \]
  Dividing through by $k$, we obtain~\eqref{eq:induction}.
\end{proof}

Combining~\eqref{eq:bla} and the claim, we obtain
\[
  \sum_{\cW \in \mathfrak U_n} \Pr\big(\cW \preceq \cG\big) \le \sum_{k \le T(n)} \left( \frac{c_H \cdot K^{w_H} n^{2 - 1/m_2(F)} }{k} \right)^{k} \le (T(n)+1) \cdot \left( \frac{c_H \cdot K^{w_H} n^{2 - 1/m_2(F)}}{T(n)} \right)^{T(n)},
\]
where the last inequality follows from the fact that the function $x \mapsto
(c/x)^x$ is increasing for $0 < x \leq c/e$ and that $T(n) 
= e_H(r-1)Cn^{2-1/m_2(F)}
\leq c_H\cdot K^{w_H}
n^{2-1/m_2(F)}/e$ if $K$ is large. This yields
\[
  \sum_{\cW \in \mathfrak U_n} \Pr\big(\cW \preceq \cG\big)
  \leq 
  \left(\frac{2c_H K^{w_H}}{e_H (r-1) C}\right)^{e_H(r-1)C n^{2-1/m_2(F)}} \leq
  e^{O(\log K) n^{2 - 1/m_2(F)}},
\]
which, together with~\eqref{eq:janson}, implies that $\Ex[Z] \to 0$, provided that $K$ is sufficiently large. This completes the proof.\qed

\section{Concluding remarks}

While, even before our work, the $1$-statement of Conjecture~\ref{conj:kk} was known to be true up to a $\log n$ factor, the situation with the $0$-statement is quite different. So far it has only been verified in the case where all the graphs $F_1, \dotsc, F_r$ are either cycles~\cite{KohKre97} or complete graphs~\cite{marciniszyn2009asymmetric}. The general case, however, seems to be rather difficult. A criterion which reduces the $0$-statement of Conjecture~\ref{conj:kk} to a purely deterministic question, a potentially fruitful approach, was given in~\cite{gugelmann2017symmetric}. We now present this reduction.

Given graphs $F_1$ and $F_2$, let $\bF(F_1, F_2)$ be the family of all graphs $F$ with the following property: There exists a copy $F_2'$ of $F_2$ in $F$ and an edge $e_0 \in E(F_2')$ such that, for each $e \in E(F_2') \setminus \{e_0\}$, there is a copy $F_1^e$ of $F_1$ in $F$ containing $e$ and 
\[
  E(F) = E(F_2') \cup \bigcup_{e \in E(F_2') \setminus \{e_0\}} E(F_1^e);
\]
we shall call such an $e_0$ an \emph{attachment edge}. Note that the graphs
$F_1^e$ need not be disjoint (in fact they are even not required to be
distinct). Intuitively, every graph in $\bF(F_1, F_2)$ is formed from a copy of
$F_2$ by gluing copies of $F_1$ on all of its edges except some edge $e_0$. Let
us call a graph $F \in \bF(F_1, F_2)$ \emph{generic} if every $F_1^e$
intersects $F_2'$ only in the edge $e$ (and no vertices other than the
endpoints of $e$) and its remaining vertices are disjoint from all the other
$F_1^{e'}$ with $e' \neq e$. Note that there
can be up to $e(F_2)\cdot e(F_1)^{e(F_2)-1}$ different generic graphs.

The main property we require from the family $\bF(F_1, F_2)$ is that these
generic graphs are the `sparsest' among all graphs in $\bF(F_1,F_2)$. In
particular, we say that $\bF(F_1, F_2)$ is \emph{asymmetric-balanced} if the
following two conditions are met for every $F \in \bF(F_1, F_2)$ and every $H \subseteq F$ with $V(H) \subsetneq V(F)$ containing an attachment edge:
 \begin{enumerate}[label=(\arabic*)]
 \item
   We have
   \[
     \frac{e(F) - e(H)}{v(F) - v(H)} \ge m_2(F_1, F_2).
   \]
 \item
   Moreover, if
  \[
    \frac{e(F) - e(H)}{v(F) - v(H)} = m_2(F_1, F_2),
  \]
     then $F$ is generic and $H$ contains a single edge (the attachment edge).
\end{enumerate}

\begin{thm}[\cite{gugelmann2017symmetric}]
  \label{thm:0_det}
  Let $F_1$ and $F_2$ be graphs such that $m_2(F_1) \ge m_2(F_2) > 1$ and suppose that the following holds:
  \begin{enumerate}[label=(\roman*)]
    \item $F_1$ and $F_2$ are strictly $2$-balanced,
    \item $F_1$ is strictly balanced w.r.t.\ $m_2(\cdot, F_2)$, 
    \item $\bF(F_1, F_2)$ is asymmetric-balanced, and
    \item for every graph $G$ such that
    \[
      \max_{G' \subseteq G} \frac{e(G')}{v(G')} \le m_2(F_1, F_2)
    \]
    we have $G \nrightarrow (F_1, F_2)$.
  \end{enumerate}
  Then there exists $c > 0$ such that if $p \le cn^{-1/m_2(F_1, F_2)}$, then
  \[
    \lim_{n \to \infty} \Pr\big( G(n,p) \to (F_1, F_2) \big) = 0.
  \]
\end{thm}

A minor modification of the proof of Theorem~\ref{thm:0_det} shows that one can further drop the requirement in (i) that $F_1$ is strictly $2$-balanced. Therefore, in order to prove the $0$-statement of Conjecture~\ref{conj:kk}, it is enough to consider a strictly $2$-balanced subgraph $F_2' \subseteq F_2$ with $m_2(F_2') = m_2(F_2)$ and a subgraph $F_1' \subseteq F_1$ that is strictly balanced w.r.t\ $m_2(\cdot, F_2')$ and satisfies $m_2(F_1', F_2') = m_2(F_1, F_2')$, and show that conditions (iii) and (iv) in Theorem~\ref{thm:0_det} hold. As an exercise, we invite the reader to show this in the case where $m_2(F_1) = m_2(F_2)$ (for part (iv) see, e.g., the appendix of~\cite{nenadov2016short}). In this case, it turns out that if $F_1' \subseteq F_1$ is chosen in the manner described above, then it is also strictly $2$-balanced and $m_2(F_1') = m_2(F_2')$; in particular, one can use Theorem~\ref{thm:0_det} without any modifications. Unfortunately, the general case remains wide open.

Finally, let us mention that the proof of Theorem~\ref{thm:main} transfers to
the setting of uniform hypergraphs with almost no changes. However, unlike for
graphs, in the case of hypergraphs of uniformity larger than two, even in the
symmetric case (i.e., $F_1 = \dotsb = F_r$) a complete characterisation of the
threshold functions is not known. We refer the interested reader
to~\cite{gugelmann2017symmetric,nenadov2017algorithmic} for further details.

\bibliography{asymmetric_ramsey}
\bibliographystyle{abbrv}

\end{document}